\documentclass[11pt]{amsart}

\usepackage{epsfig}  		
\usepackage{epic,eepic}       
\usepackage{amsmath}
\usepackage[retainorgcmds]{IEEEtrantools}
\usepackage{graphicx}
\usepackage[margin=1in]{geometry}
\usepackage{amsfonts}
\usepackage{amsthm}
\usepackage{setspace}
\usepackage{mathtools}
\usepackage{listings}
\usepackage{mathabx}
\usepackage{amssymb}
\usepackage[mathscr]{euscript}
\usepackage{float}
\newtheorem{thm}{Theorem}
\theoremstyle{definition}
\newtheorem{lem}{Lemma}
\newtheorem{dfn}{Definition}

\newtheorem{prp}{Proposition}
\newtheorem{cnj}{Conjecture}

\newcommand{\im}{\text{im}}

\newcommand{\Z}{\mathbb{Z}}

\newcommand{\R}{\mathbb{R}}

\newcommand{\C}{\mathbb{C}}

\author{Cole Hugelmeyer} 
\title{Inscribed squares and relation avoiding paths.}

\begin{document}

\maketitle 

\begin{abstract}We develop a connection between the inscribed square problem and the question of understanding relation avoiding paths in a complex vector space. Our main theorem is that a Jordan curve with no inscribed squares would have a seemingly impossible structure which we call a square envelope. We will make some conjectures about the nature of relation avoiding paths in vector spaces, and show that these conjectures would imply the existence of inscribed squares in Jordan curves with finitely many arbitrarily complicated singularities.\end{abstract}

\section{Introduction}

A Jordan curve is a continuous injective function $\gamma: S^1\to \R^2$ which wraps counterclockwise around the region it encloses. An inscribed square of $\gamma$ is a quadruple of distinct points on $\im(\gamma)$ which form a square in $\R^2$.

\begin{figure}[H]
\centering
\includegraphics[scale = 0.7]{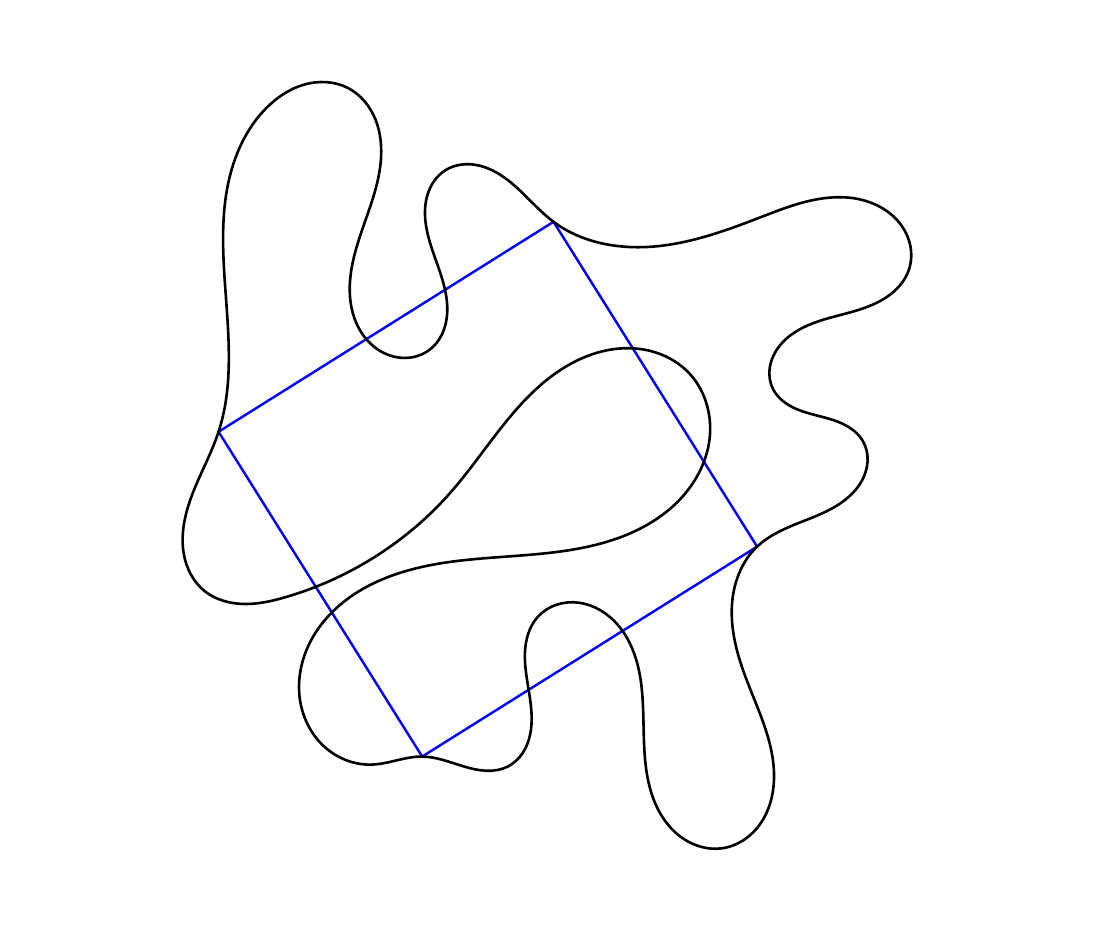}
\caption{An inscribed square of a Jordan curve.}
\end{figure}

\begin{cnj}[Toeplitz 1911]
Every Jordan curve has an inscribed square.
\end{cnj} 

Towards this end, we define a \emph{bad Jordan curve} to be a Jordan curve with no inscribed squares. Our main result is that a bad Jordan curve must have a structure which we call a \emph{square envelope}. A square envelope is, loosely speaking, a square moving with time, so that its first two corners are always outside the Jordan curve, its second two corners are always inside the Jordan curve, and the outside corners wrap completely around the Jordan curve as time progresses. 

To rigorously define a square envelope, we let $T:\R^2\to \R^2$ be the linear map corresponding to 90 degree counterclockwise rotation. For a pair of points $a$ and $b$ in the plane, we define $S_1(a,b) = a + T(b-a)$ and $S_2(a,b) = b + T(b-a)$. If $a\neq b$, then the four points $a,b,S_2(a,b),S_1(a,b)$ form the vertices of a square, labeled counterclockwise.

\begin{dfn}
Let $\gamma: S^1\to \R^2$ be a Jordan curve. A \emph{square envelope} of $\gamma$ is a pair of continuous functions mapping from $\R$ to the plane, $e_1, e_2: \R\to \R^2$, such that all of the following are true.

\begin{itemize}
\item[1)] For all $x\in \R$, $e_1(x)$ and $e_2(x)$ are in the open exterior of $\gamma$.
\item[2)] For all $x\in \R$, $S_1(e_1(x),e_2(x))$ and $S_2(e_1(x), e_2(x))$ are in the open interior of $\gamma$.
\item[3)] $\lim_{x\to \infty} \|e_2(x) - e_1(x)\| = \lim_{x\to -\infty} \|e_2(x) - e_1(x)\| = 0$.
\item[4)] Let $\ell_\lambda$ be the closed curve obtained by first following $e_1$ from $e_1(-\lambda)$ to $e_1(\lambda)$, then a straight line from $e_1(\lambda)$ to $e_2(\lambda)$, then $e_2$ from $e_2(\lambda)$ to $e_2(-\lambda)$, and finally a straight line from $e_2(-\lambda)$ back to $e_1(-\lambda)$.  If $p$ is any point inside $\gamma$, then there exists a real number $\lambda_0 > 0$ such for all $\lambda > \lambda_0$, we have that the winding number of $\ell_\lambda$ around $p$ is equal to one. 
\end{itemize}
\end{dfn}

\begin{figure}[H]
\centering
\includegraphics[scale = 0.5]{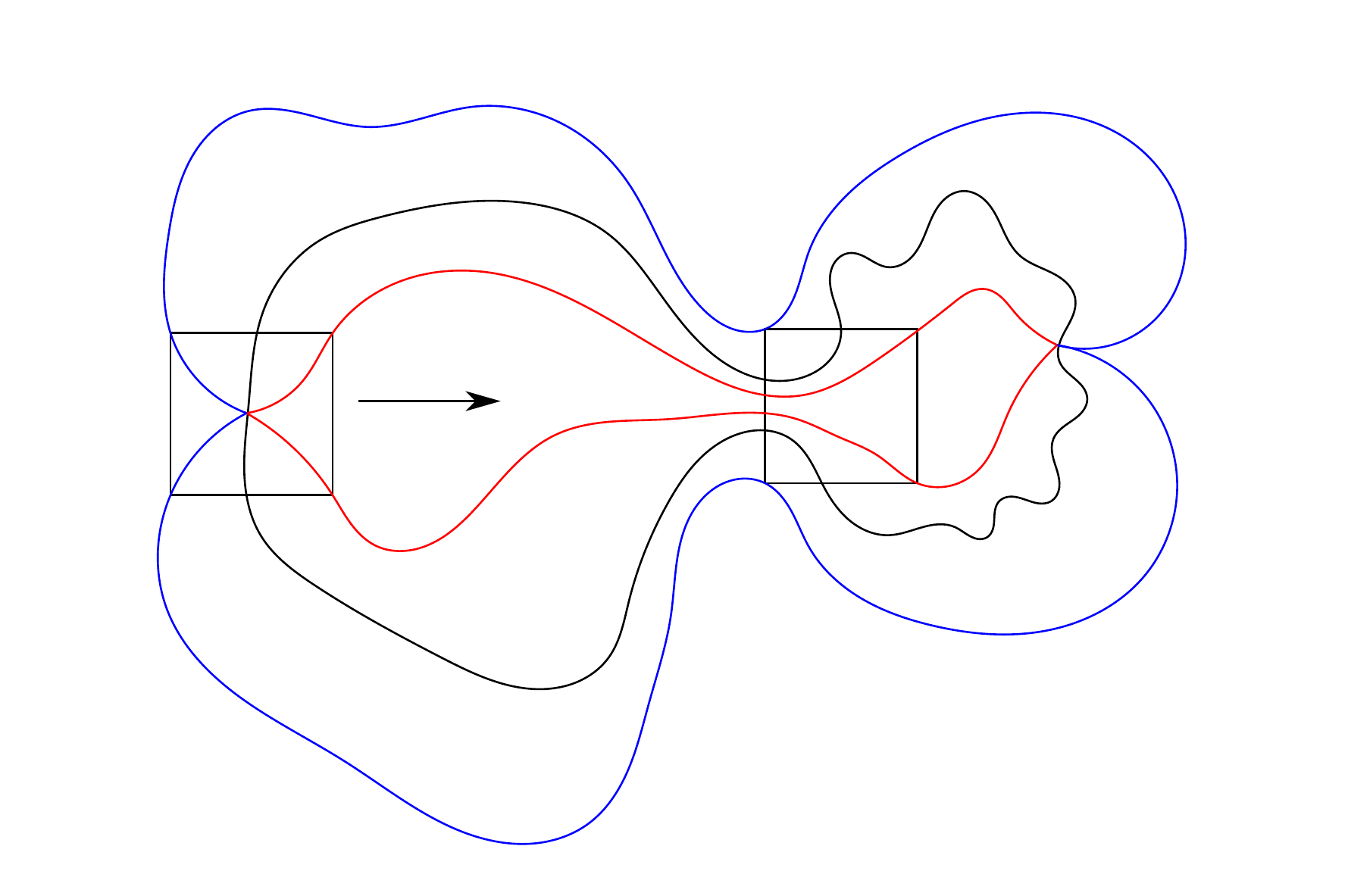}
\caption{As square envelopes conjecturally do not exist, they are somewhat difficult to draw. This is a rough approximation of what one might look like. The square must somehow have its first two corners move around the outside of the Jordan curve while the other two corners remain inside.}
\end{figure}

\begin{thm}
Every bad Jordan curve has a square envelope.
\end{thm}

This theorem allows us to make a connection between the inscribed square problem and the problem of understanding relation avoiding paths in a vector space. The set of squares is the vector space, and the linear relation we avoid is the first two corners of one square touching the second two corners of another.

Let $V$ be a vector space over $\C$, and suppose $R_1,R_2,...,R_n$ are linear relations, subspaces of the vector space $V\times V$. We then let $R$ be the relation given by the union of sets $\cup_{i = 1}^n R_i$. We will assume that $R$ is symmetric, namely $aRb\iff bRa$. 

A \emph{relation-avoiding path} is a continuous function $p:[0,1]\to V$ such that there does not exist any pair of times $t_1,t_2$ in $[0,1]$ with $p(t_1)Rp(t_2)$. We define a \emph{relation-avoiding origin path} to be a continuous function $p: [0,\infty)\to V$ with $\lim_{t\to \infty}p(t) = 0$, such that there does not exist any pair of times $t_1,t_2$ in $[0,\infty)$, for which $p(t_1)Rp(t_2)$. The origin is not included as the endpoint of this path because we always have $0R0$. 

\begin{dfn}
We say two relation avoiding origin paths $p$ and $q$ are \emph{weakly homotopic} if there exists a continuous function $h: [0,1]\times [0,\infty) \to V$ such that 
\begin{itemize}
\item[1)]  $\lim_{t\to \infty}h(s,t) = 0$ for all $s$.
\item[2)]  $h(0,t) = p(t)$ for all $t$.
\item[3)]  $h(1,t) = q(t)$ for all $t$. 
\item[4)]  There does not exist any $(s,t)\in  [0,1]\times [0,\infty)$ with $h(s,t)\,R\,h(s,0)$.
\end{itemize} 
We say $p$ and $q$ are \emph{strongly homotopic} if $h$ can be chosen so that $h(s,-)$ is a relation avoiding origin path for all $s$.
\end{dfn}

We propose the following conjecture about relation avoiding origin paths. 

\begin{cnj}[SC, Spiral Conjecture]
Every relation avoiding origin path $p$ is weakly homotopic to a relation avoiding origin path $q$ for which there exist linearly independent vectors $x_1,...,x_n$ in $V$ and complex numbers $a_1,...,a_n$ such that $$q(t) = \sum_{i = 1}^n x_i\cdot e^{a_it}$$ for all $t$.
\end{cnj}

The following fact provides some evidence for this conjecture.

\begin{thm}
The spiral conjecture is true when $V$ is one dimensional.
\end{thm}

Finally, we will prove the following implication of Conjecture 2.

\begin{thm}
Assume the spiral conjecture. Then any Jordan curve $\gamma$ which is smooth except at finitely many points has an inscribed square. The Jordan curve is permitted to be arbitrarily complicated near these finitely many singularities. 
\end{thm}

We propose that delving into the study of relation avoiding origin paths could be a potential approach towards solving the inscribed square problem. At the end of the paper, we will make more conjectures about relation avoiding origin paths, which we hope might guide future research.

\section{Square Envelopes}
 
If $\gamma$ is a bad Jordan curve, we will define an integer $b_\gamma$ which we call the bad wrapping number of $\gamma$. This integer is only well-defined for bad Jordan curves. By counting inscribed squares in a generic approximation of $\gamma$, we will prove that $b_\gamma$ is odd. Then, we will use this fact to construct a square envelope of the Jordan curve. 

Let $\gamma$ be a Jordan curve. We may choose a continuous function $f: \R^2\to \R$ with the property that if $a$ is in the interior of $\gamma$, then $f(a) < 0$, if $a$ is on the image of $\gamma$, then $f(a) = 0$, and if $a$ is outside of $\gamma$, then $f(a) > 0$. Using this function, we can define a function $g_{f,\gamma}: S^1\times S^1\to \R^2$, by the formula $$ g_{f,\gamma}(x,y) = (f(S_1(\gamma(x),\gamma(y))),f(S_2(\gamma(x),\gamma(y)))).$$

If $\gamma$ is a bad Jordan curve, then $g_{f,\gamma}(x,y) = (0,0)$ if and only if $x = y$, because if we have $x\neq y$ and $g_{f,\gamma}(x,y) = (0,0)$, then $\gamma(x), \gamma(y), S_2(\gamma(x),\gamma(y)),S_1(\gamma(x),\gamma(y)))$ form the vertices of an inscribed square.

Let $C = \{(x,y)\in S^1\times S^1: x\neq y\}$. Topologically, $C$ is an open cylinder. By the above proposition, we see that if $\gamma$ is a bad Jordan curve, then we can restrict the domain of $g_{f,\gamma}$ to $C$ to get a map $$g_{f,\gamma}|_{C} : C\to \R^2\setminus\{(0,0)\}.$$

We now fix homotopy equivalences between $C$, $\R^2\setminus \{(0,0)\}$, and $S^1$. Our homotopy equivalence $C\to S^1$ is to simply take the first coordinate of the ordered pair. Our homotopy equivalence $\R^2\setminus \{(0,0)\}\to S^1$ is radial projection onto the unit circle, which we orient counterclockwise. 

Using these homotopy equivalences, we see that, $g_{f,\gamma}|_{C}$ induces a map $S^1\to S^1$ up to homotopy, and therefore gives us an element of $\pi_1(S^1)\simeq \Z$. This integer is defined to be $b_\gamma$.

$b_\gamma$ does not depend on the choice of $f$, because given two choices of $f$, one can be continuously transformed into the other while always remaining a valid choice of $f$. This induces a homotopy between the resulting maps $S^1\to S^1$, which implies that the value of $b_\gamma$ is constant.

\begin{lem}
For any bad Jordan curve, $b_\gamma$ is odd.
\end{lem}

The proof of this lemma will be given later in this section. The main idea behind the proof is that each inscribed square of a generic smooth approximation of $\gamma$ contributes parity to $b_\gamma$ in a way that depends on the cyclic order of the vertices of the square on the Jordan curve. The parity of inscribed squares with a given cyclic ordering is independent of the Jordan curve, so we can simply calculate the parity of $b_\gamma$ by summing the contributions from each square type. The result is odd parity.

Lemma 1 can be used to prove Theorem 1, which we will do at the end of this section. The main idea is that because $b_\gamma \neq 0$, the map $g_{f,\gamma}:C\to \R^2\setminus\{(0,0)\}$ wraps nontrivially around the origin. This implies that there must be a path between the two ends of the cylinder $C$ that maps into the lower left quadrant of the plane under $g_{f,\gamma}$. This path is almost a square envelope, except that $e_1$ and $e_2$ are on the Jordan curve rather than being inside its open exterior. This can be remedied by simply pushing the paths off of the curve slightly.\\

For a manifold $M$, let $\tilde{C}_4(M)$ denote the manifold of cyclically ordered quadruples of distinct points in $M$. In other words, if $C_{4\;\text{dist}}(M)$ is the subset of $M^4$ consisting of ordered quadruples of distinct points, then $\tilde{C}_4(M)  = C_{4\;\text{dist}}(M)/(\Z/4\Z)$, where $\Z/4\Z$ acts by cyclic permutation of the entries of the 4-tuple. Let $Sq$ denote the the submanifold of $\tilde{C}_4(\R^2)$ consisting of quadruples of points which form squares labeled counterclockwise.

We call a Jordan curve $\gamma: S^1\to \R^2$ \emph{generic} if, within $\tilde{C}_4(\R^2)$, the subspace $\tilde{C}_4(\im(\gamma))$ intersects transversely with $Sq$.  

Let $h: S^1\times[0,1]\to \R^2$ be a smooth homotopy between two generic Jordan curves for which $h(\cdot,t)$ is a smooth embedding for all $t\in [0,1]$. We say that $h$ is a \emph{generic homotopy} if, within $\tilde{C}_4(\R^2)\times[0,1]$, the subspace $\{(q,t): q\in \tilde{C}_4(\im(h(\cdot,t))\}$ intersects transversely with the subspace $Sq\times[0,1]$.

By the transversality theory found in \cite{cantarella}, we have the following facts about generic Jordan curves.

\begin{itemize}
\item[1)] Any Jordan curve can be approximated arbitrarily well in the $C^0$ topology by generic Jordan curves.
\item[2)] Any two generic Jordan curves have a generic homotopy between them. 
\item[3)] For a generic homotopy $h$, the intersection $\{(q,t): q\in \tilde{C}_4(\im(h(\cdot,t))\} \;\cap\;Sq\times[0,1]$ is a compact 1-manifold with boundary. The boundary of this manifold consists of the inscribed squares of the two generic Jordan curves at the beginning and end of the homotopy.
\end{itemize}

Note that the manifold $\tilde{C}_4(S^1)$ has three connected components. This corresponds to three types of inscribed square.

\begin{dfn}
Let $\gamma$ be a Jordan curve. An inscribed square of $\gamma$ is called type I if the counterclockwise order of points on the square is the same as that of the Jordan curve. Such squares are also called gracing squares. A square is called type II if when we label the vertices of the square as $1,2,3,4$ in counterclockwise order, these vertices appear in the order $1,3,2,4$ when we go counterclockwise along the Jordan curve. Finally, a square is called type III if the counterclockwise order of the vertices is opposite to the counterclockwise order of the Jordan curve. 
\end{dfn}

\begin{figure}[H]
\centering
\includegraphics[scale = 0.4]{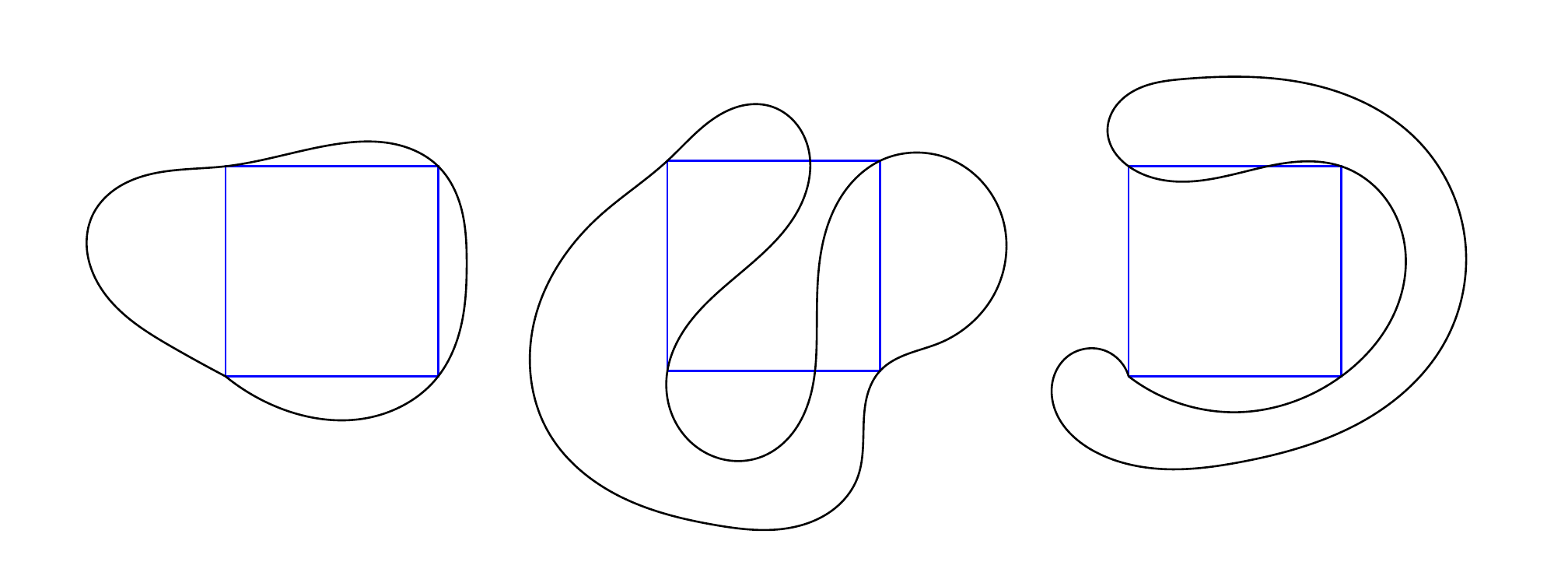}
\caption{Inscribed squares of type I, II, and III respectively.}
\end{figure}

\begin{prp}
If $\gamma$ is a generic Jordan curve, then of the inscribed squares of $\gamma$, an odd number of them are type I, an even number are type II, and an even number are type III.
\end{prp}

\begin{proof}
There are three connected components of $\tilde{C}_4(\im(\gamma))$ corresponding to the three types of inscribed squares when we intersect with $Sq$. This means that if $h$ is a generic homotopy between generic Jordan curves, then the 1-manifold $\{(q,t): q\in \tilde{C}_4(\im(h(\cdot,t))\} \;\cap\;Sq\times[0,1]$ can be separated into three distinct clopen pieces, one for each type of inscribed square. This proves that the parity of each type of inscribed square is invariant of which generic Jordan curve we choose. Finally we can compute the parities by finding the inscribed squares of any generic Jordan curve.  The quintessential generic Jordan curve is a non-circular ellipse. This curve has one gracing square, but no inscribed square of type II or III. 
\end{proof}

Let $\gamma$ be a generic Jordan curve, and let $f$ be a smooth function $\R^2\to \R$ which has $\gamma$ as its zero set, is negative inside of $\gamma$, positive outside of $\gamma$, and has non-vanishing gradient on $\gamma$. Since $\gamma$ is a smooth embedding, such an $f$ always exists. Now, we will consider $g_{f,\gamma}: (S^1)^2 \to \R^2$, as defined in the previous section. 

We see that the zeros of $g_{f,\gamma}$ consist of the diagonal $\Delta = \{(a,a): a\in S^1\}$, as well as four points for each inscribed square in $\tilde{C}_4(\im(\gamma))\cap Sq$, one point for each side of the square. Furthermore, at these zeros corresponding to inscribed squares, $g_{f,\gamma}$ has nonsingular derivative because $\gamma$ is generic and $f$ has non-vanishing gradient on $\im(\gamma)$.  Let $X = \{p\in (S^1)^2: g_{f,\gamma}(p) \neq (0,0)\}$. Let $\omega_\gamma$ be the cohomology class in $H^1(X;\Z/2\Z)$ given by pulling back the nontrivial class of $H^1(\R^2\setminus \{(0,0)\};\Z/2\Z)$ under $g_{f,\gamma}|_X$. A small loop around any of the zeroes corresponding to sides of inscribed squares will evaluate nontrivially under $\omega$ because $g_{f,\gamma}$ has nonsingular derivative at these zeroes. Furthermore, a loop obtained by pushing the diagonal off to one side or the other will evaluate trivially under $\omega_\gamma$ because it will correspond to a small square sliding around the Jordan curve with two corners on $\gamma$ and the other two corners off to one side of $\gamma$, so the corresponding loop in $\R^2\setminus \{(0,0)\}$ will remain in just one quadrant of the plane and therefore cannot wrap around the origin.

Putting this all together, we see that if $\ell$ is a simple closed curve in $X$ which has the homotopy class of the diagonal in $S^1\times S^1$, then $\omega_\gamma(\ell)$ is equal to the parity of the number of zeroes of $g_{f,\gamma}$ in either connected component of $(S^1)^2\setminus(\Delta \cup \ell)$.

Let $\Delta'$ be the anti-diagonal of $(S^1)^2$, the loop consisting of pairs $(a,b)$ where $a$ and $b$ are antipodes. Then $(S^1)^2\setminus(\Delta\cup \Delta')$ has two connected components, which we call $A$ and $B$. The component $A$ is the set of pairs $(a,b)$ for which the angle from $a$ to $b$ is less than $\pi$ and the component $B$ is the set of pairs for which the angle from $a$ to $b$ is greater than $\pi$.

If $\gamma$ is a bad Jordan curve, and $\gamma'$ is a sufficiently $C^0$-close generic approximation, then the parity of $b_\gamma$ is equal to $\omega_{\gamma'}(\Delta')$. We can therefore compute $b_\gamma$ by counting the zeroes of $g_{f,\gamma'}$ in $A$.

\begin{prp}
If $\gamma$ is a bad Jordan curve, and $\gamma'$ is a sufficiently $C^0$-close generic approximation, then each type I inscribed square of $\gamma'$ has exactly three of its corresponding zeroes of $g_{f,\gamma'}$ in $A$, each type II inscribed square has exactly two of its corresponding zeroes in $A$, and each type III inscribed square has exactly one of its corresponding zeroes in $A$.
\end{prp}

\begin{proof}
Let $\gamma_1,\gamma_2,\gamma_3,...$ be a sequence of generic Jordan curves that limit to our bad Jordan curve $\gamma$ in the $C^0$ topology. Let $s_n$ be the side length of the largest inscribed square of $\gamma_n$. We know that $\lim_{n\to \infty}s_n = 0$ because otherwise there would be a convergent subsequence of squares that limit to an inscribed square of our bad Jordan curve. We can give $S^1$ a metric by identifying it with the unit circle in $\R^2$ and using euclidean distance. Using this identification, we let $M$ be the set of all pairs $(a,b)\in (S^1)^2$ which have $\|a-b\| \geq 1$. Let $N\in \Z$ be sufficiently large that both $\sup_{x\in S^1}\|\gamma_N(x) - \gamma(x)\|$ and $s_N$ are smaller than the quantity $\varepsilon = \frac{1}{4} \cdot \min_{(a,b)\in M} \|\gamma(a) - \gamma(b)\|$. We claim that $\gamma' = \gamma_N$ is an approximation for which the proposition holds. 

Let $a,b,c,d$ be the vertices of an inscribed square of $\gamma_N$, labeled counterclockwise. Let $Q$ be the set of four points in $S^1$ which map onto $\{a,b,c,d\}$ under $\gamma_N$. We claim that the diameter of $Q$ is less than one. To prove this, suppose there were elements $x$ and $y$ in $Q$ with $\|x-y\|\geq 1$. Then $(x,y)\in M$, so $\|\gamma(x) - \gamma(y)\| \geq 4\varepsilon$, so $\|\gamma_N(x) - \gamma_N(y)\|\geq 2\varepsilon$. However, $2\varepsilon >  \sqrt{2} \cdot  s_N$ and $\gamma_N(x)$ and $\gamma_N(y)$ are vertices of a square with side length at most $s_N$, so this is impossible. 

Since the diameter of $Q$ is less than $1$, all four points of $Q$ must lie within some interval $I$ of angular length $\pi/3$ radians. Orient $I$ counterclockwise, and let $w,x,y,z$ be the four elements of $Q$ in the order they appear along $I$. We know that $(w,x),(w,y),(w,z),(x,y),(x,z),$ and $(y,z)$ are in $A$, and all the other ordered pairs are in $B$.

We know that $\gamma_N$ maps $\{w,x,y,z\}$ onto $\{a,b,c,d\}$, but it might not preserve the ordering. Without loss of generality, we can assume $\gamma_N(w) = a$ because we can permute the labels $a,b,c,d$ cyclically. Then, we have six possibilities to check for the six permutations of the remaining three letters.  We will now check all of the possibilities. 

\begin{itemize}
\item[1)] If $(w,x,y,z) \mapsto (a,b,c,d)$, the square is type I and the corresponding zeroes in $A$ are $(w,x),(x,y),(y,z)$.
\item[2)] If $(w,x,y,z) \mapsto (a,c,b,d)$, the square is type II and the corresponding zeroes in $A$ are $(w,y),(x,z)$.
\item[3)] If $(w,x,y,z) \mapsto (a,b,d,c)$, the square is type II and the corresponding zeroes in $A$ are $(w,x),(x,z)$.
\item[4)] If $(w,x,y,z) \mapsto (a,d,c,b)$, the square is type III and the only corresponding zero in $A$ is $(w,z)$.
\item[5)] If $(w,x,y,z) \mapsto (a,d,b,c)$, the square is type II and the corresponding zeroes in $A$ are $(w,y),(y,z)$.
\item[6)] If $(w,x,y,z) \mapsto (a,c,d,b)$, the square is type II and the corresponding zeroes in $A$ are $(w,z),(x,y)$.
\end{itemize}

This confirms the proposition. The type I squares have three corresponding zeroes in $A$, the type II squares have two corresponding zeroes in $A$, and the type III squares have only one corresponding zero in $A$.

\end{proof}

We can now prove Lemma 1 and Theorem 1. 

\proof[Proof of Lemma 1.]{
Let $\gamma$ be a bad Jordan curve, let $\gamma'$ be a sufficiently $C^0$-close generic approximation, and let $f$ be a function with $\gamma'$ as its zero set as above. We now count zeroes of $g_{f,\gamma'}$ in $A$, the parity of which is the parity of $b_\gamma$. We count the zeroes by those corresponding to the squares of each type. We have an odd number of threes, an even number of twos, and an even number of ones. This totals to an odd number.
}

\proof[Proof of Theorem 1.]{
We wish to prove that every bad Jordan curve has a square envelope. We have the function $g_{f,\gamma}|_C: C\to \R^2\setminus\{(0,0)\}$, which we know to be homotopically nontrivial. Let $\pi: \R^2\setminus\{(0,0)\}\to S^1$ be radial projection onto the unit circle, and let $r = \pi\circ g_{f,\gamma}|_C$. We now take $r'$ to be a smooth approximation of $r$ with $\|r(x) - r'(x)\| < 1/4$ for all $x\in C$. Let $u\in S^1$ be a regular value for $r'$ within distance $1/4$ of the point $(-1/\sqrt 2,-1/\sqrt 2)$. Then $r^{-1}(u)$ is a 1-dimensional manifold, $L$, which is mapped under $r$ into a circle of radius $1/2$ around $(-1/\sqrt 2,-1/\sqrt 2)$, which is entirely in the lower left quadrant of the plane. Therefore, $g_{f,\gamma}|_C$ maps $L$ into the lower left quadrant of the plane. This implies that the square corresponding to a point of $L$ has its first two corners on $\gamma$, and the other two in the open interior of $\gamma$. Furthermore, since $b_\gamma$ is odd, we see that $L$ must have an odd number of connected components which are homeomorphic to $\R$ with the two ends limiting to the two boundaries of $C$. Parameterizing such a connected component, we have functions $e_1,e_2: \R\to \R^2$ with $\im(e_1) \cup \im(e_2) \subseteq \im(\gamma)$, and $\lim_{x\to \pm\infty} \|e_1(x)-e_2(x)\| = 0$, and with the property that $S_1(e_1(x),e_2(x))$ and $S_2(e_1(x),e_2(x))$ are always in the open interior of $\gamma$. Since the interior of gamma is open, we may choose a continuous function $\varepsilon: \R\to \R_{>0}$ with the property that, for all $x$, if $\|a - e_1(x)\|$ and $\|b - e_2(x)\|$ are less than $\varepsilon(x)$, then $S_1(a,b)$ and $S_2(a,b)$ are in the open interior of $\gamma$ as well.  Thus, if we choose a continuous deformation of $(e_1,e_2)$ in which we push $e_1(x)$ and $e_2(x)$ off of the Jordan curve without exceeding a distance of $\varepsilon(x)$, we obtain a square envelope for $\gamma$. 
}

\section{Relation avoiding paths}

In this section, we will prove that if the spiral conjecture is true, then the inscribed square conjecture holds for Jordan curves which are smooth except at a finite number of arbitrarily complicated singularities. 

However, we will first present a proof of Theorem 2, that the spiral conjecture is indeed true in the one-dimensional case.

\begin{proof}[Proof of Theorem 2.]
Identifying $V$ with $\C$, our relation avoiding origin paths are paths to the origin in the complex plane $p: [0,\infty) \to \C$. Furthermore, we know that $p(t)$ is never zero, because $(0,0)$ is in every linear relation. Thus, without loss of generality, we can ignore any of the $R_i$ corresponding to $\{0\}\times \C$ or $\C\times \{0\}$. We can therefore rewrite the relation $R$ as $$ xRy \iff \exists i\in \{1,...,n\} \;\;\;  x = \alpha_i y \;\; \text{or} \;\; y = \alpha_i x$$ where $\alpha_1,...,\alpha_n$ are nonzero complex numbers with norm at most one. In particular, our path $p$ is relation avoiding if and only if it is disjoint from the paths $\alpha_1 p, \alpha_2 p,..., \alpha_n p$. Without loss of generality, we can assume that $p(0) = 1$, and that $|p(t)|\leq 1$ for all $t$. The reason we can do this is that we can always homotope $p$ to such a path by first contracting it within its image so that the starting point is at the point with maximal norm, and then rescaling within $\C$ so that the starting point becomes $1$. We can also assume without loss of generality that $p$ follows the path of a logarithmic spiral inside some small neighborhood around 1, because making this happen only requires a small perturbation. Working with these assumptions, we fix a branch of the logarithm, and let $\ell_1,...,\ell_n$ be paths in upper-half plane such that $e^{2\pi i \ell_i(t)} = \alpha_ip(t)$ with $\ell_i(0) = \frac{1}{2\pi i}\ln(\alpha_i)$, and similarly let $\ell$ be a lift of $p$ with $\ell(0) = 0$. The path $\ell$ divides the upper-half plane, and we define $U$ to be the set of all points of $\mathbb{H}\setminus \im(\ell)$ from which a path to $-\infty$ has even intersection parity with $\ell$. We then define integers $k_1,...,k_n$, where $k_i$ is the largest integer such that $\ell_i(0) + k_i \in U$.

We see that $\ell$ is disjoint from $\ell_i + k$ for any index $i$ and integer $k$, and this gives us paths disjoint from $\ell$, going up to $i\infty$, from every point of the form $\ell(0) + k$. Therefore, the homotopy type of $\ell$ can be completely determined by knowing which points of the form $\ell_i(0)+k$ are in $U$.  Thus, to prove the one-dimensional spiral conjecture, it suffices to prove that there exists a $\theta\in (0,\pi)$ so that for all $i\in \{1,...,n\}$, we have the inequalities $\arg(\ell_i(0) + k_i)\geq \theta \geq \arg(\ell_i(0) + k_i + 1)$. This $\theta$ would then denote the angle of a straight line in the upper half plane that exponentiates to the desired logarithmic spiral. To prove that such a $\theta$ exists, it suffices to prove that for all $i$ and $j$ we have $\arg(\ell_j(0)+k_j+1)\leq\arg(\ell_i(0) + k_i) $.

We define a \emph{split pair} to be a pair of points in the upper half-plane $(p,q)$, such that $p + \ell$ and $q + \ell$ are both disjoint from $\ell$, with $p\in U$ and $q\not\in U$. For instance, $(\ell_i(0) + k_i,\ell_j(0) + k_j + 1)$ is always a split pair.

\begin{figure}[H]
    \item \mbox{}

    \begin{center}
        \includegraphics*[height=8cm]{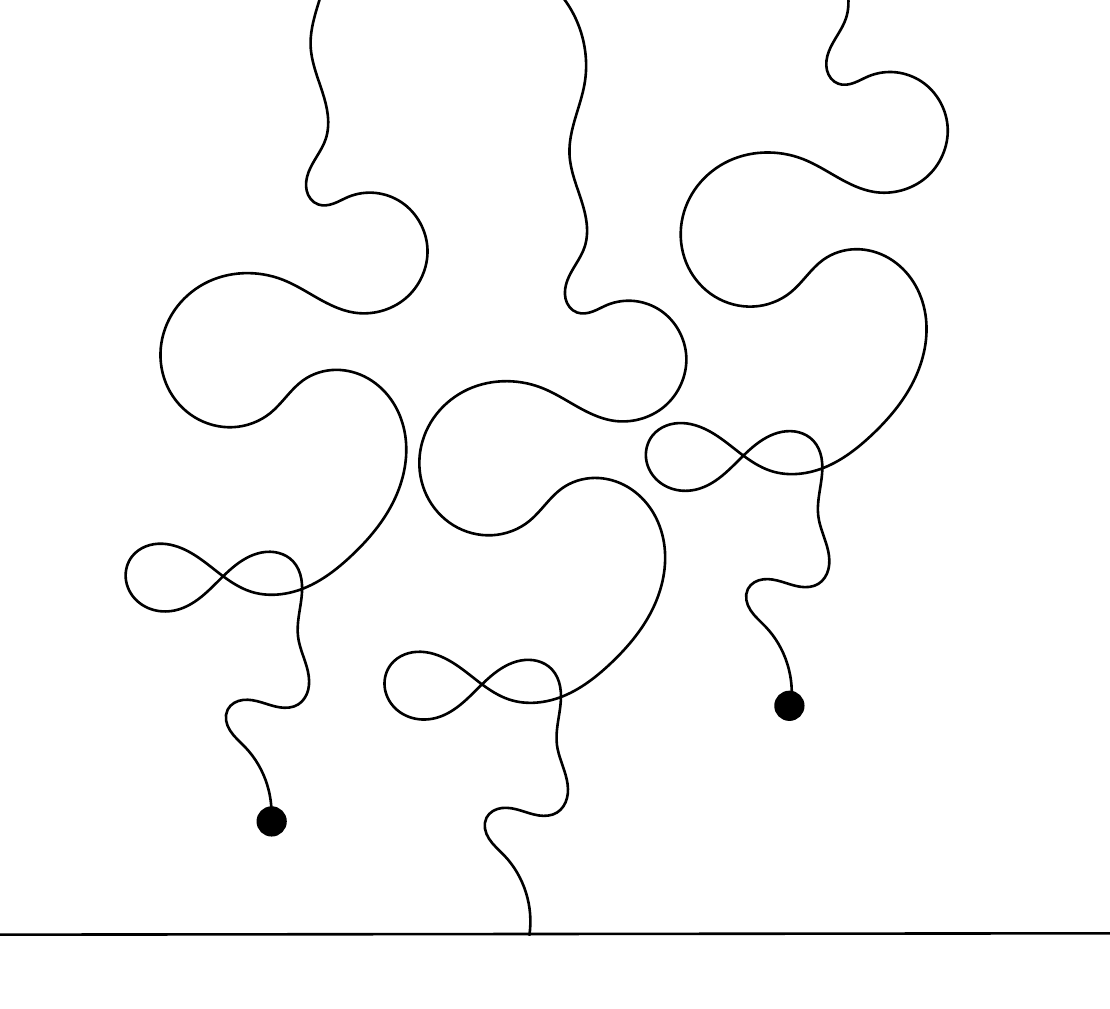}
    \end{center}
    \caption{A depiction of a split pair, disjoint translates of $\ell$, one on each side.}
\end{figure}

Given a split pair $(p,q)$, we can construct a new split pair as follows. If $\text{im}(p) \geq \text{im}(q)$, then $(p-q,q)$ is the new split pair. If $\text{im}(q) > \text{im}(p)$, then $(p,q-p)$ is the new split pair. The split pair obtained by applying this transformation is called the derived split pair. We claim that the derived split pair is always another split pair. 

\begin{figure}[H]
    \item \mbox{}

    \begin{center}
        \includegraphics*[height=5cm]{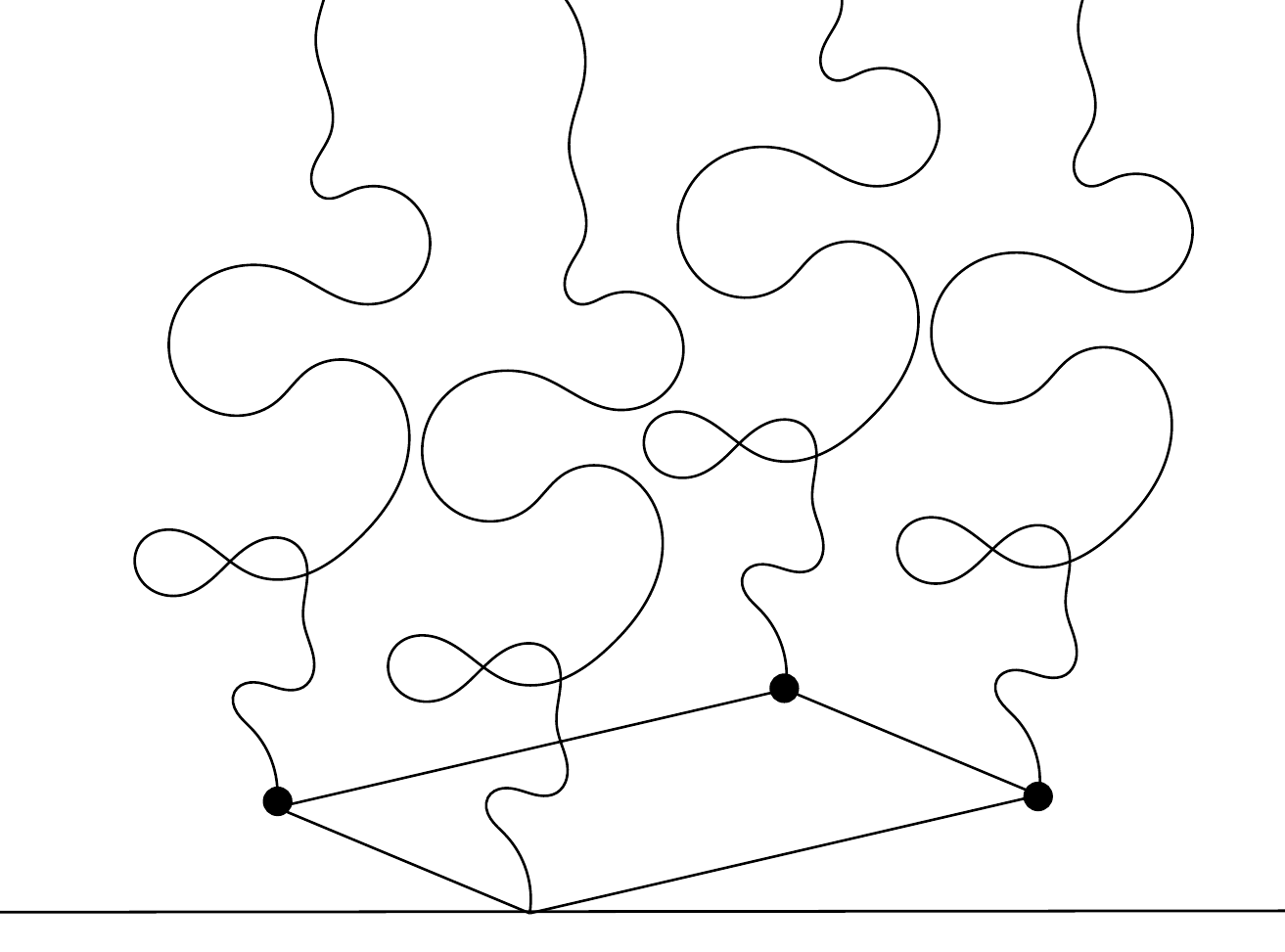}
    \end{center}
    \caption{One way to understand derived split pairs is that one draws a parallelogram spanning the origin and the two points, then moves the uppermost point to the new corner.}
\end{figure}

To prove this claim, we consider the case that $\text{im}(p) \geq \text{im}(q)$. We have that $\ell + q$ is entirely on the right side of $\ell$, and $\ell + p$ is entirely on the left side of $\ell$. Furthermore, since all of $\ell+p$ has greater imaginary coordinate than $\text{im}(q)$, we therefore have that $p\in \{a\in U: \im(a) \geq \im(q) \}\subseteq U + q$. This tells us that $p-q$ is in $U$, so $(p-q,q)$ is a valid split pair. The argument for the other case is similar. 

We say a split pair is \emph{good} if $\arg(p) > \arg(q)$. We say a split pair is \emph{bad} otherwise. We see that the derived split pair of a bad split pair is bad, and the derived split pair of a good split pair is good. To complete the proof of the theorem, it suffices to prove that all split pairs are good, so we will consider what occurs under repeated derivation of a bad split pair. First of all, note that if $\im(p)/\im(q)$ is a rational number, then we will eventually have a split pair with a real number as one of the terms, and we can see that any split pair containing a real number is good. Thus, we have shown that $\im(p)/\im(q)$ is irrational for any bad split pair. Now, we will proceed by contradiction to prove that there are no bad split pairs. Suppose that $(p,q)$ is a bad split pair. We will treat the cases of $\arg(p) < \arg(q)$ and $\arg(p) = \arg(q)$ separately. 

Suppose that $(p,q)$ is a bad split pair with $\arg(p) < \arg(q)$. Then, the $i$-th derived split pair is of the form $(a_ip - b_iq, c_iq-d_ip)$, for nonnegative integers $a_i,b_i,c_i,d_i$ such that $$\lim_{i\to \infty} a_i/b_i =  \lim_{i\to \infty} d_i/c_i = \im(q)/\im(p)$$ This tells us that while the imaginary parts of the terms of the derived split pairs approach zero, the real parts approach $+\infty$ and $-\infty$ respectively. However, this would force the second term to eventually be in $U$, which gives us a contradiction. 

Now, suppose $\arg(p) = \arg(q)$. In this case, the terms of the derived split pairs stay on a fixed straight line, and they approach zero. However, earlier in the proof, we assumed without loss of generality that $p$ followed the path of a logarithmic spiral in some small neighborhood around 1, which means $\ell$ is a straight line in some small neighborhood around zero. This means that there are no bad split pairs in this neighborhood, so we have a contradiction. This completes the proof.

\end{proof}

Before we prove Theorem 3, we need to develop some notation surrounding square envelopes, and prove a couple lemmas. 

\begin{dfn}
Let $\gamma:S^1\to \R^2$ be a Jordan curve, and let $(e_1,e_2)$ be a square envelope. We write $\sigma_{p}^*(e_1,e_2,\gamma)$, where $p\in \{1,2\}$ and $* \in \{+,-\}$, to denote a sign in $\{+1,-1\}$ determined as follows. For $\lambda\in \R$, let $P_\lambda$ be the path that starts at $e_1(\lambda)$, then goes along $e_1$ to $e_1(0)$, then goes in a straight line from $e_1(0)$ to $e_2(0)$, then follows $e_2$ to $e_2(\lambda)$, then goes in a straight line back to the starting point $e_1(\lambda)$. We then let $n(\lambda)$ be the wrapping number of $P_\lambda$ around $S_p(e_1(0),e_2(0))$. Finally, we set $\sigma_{p}^*(e_1,e_2,\gamma) = \lim_{\lambda \to *\infty}(-1)^{n(\lambda)}$. 
\end{dfn}

It is worth observing that these parities are determined by how the Jordan curve winds between the vertices of a small square near $t = \pm\infty$. For curves with finitely many singularities, they are related to each other by the following lemma.

\begin{lem}
If $\gamma$ is a Jordan curve which is smooth except at finitely many points, then the following facts are true for any square envelope $(e_1,e_2)$ of $\gamma$. 
\begin{itemize}
\item[1)] $\sigma_1^+(e_1,e_2,\gamma) = \sigma_2^+(e_1,e_2,\gamma)$ and $\sigma_1^-(e_1,e_2,\gamma) = \sigma_2^-(e_1,e_2,\gamma)$
\item[2)] $\sigma_1^+(e_1,e_2,\gamma) = -\sigma_1^-(e_1,e_2,\gamma)$
\item[3)] If $*\in \{+,-\}$ is such that $\sigma_1^*(e_1,e_2,\gamma) = -1$, then the limits  $\lim_{t\to *\infty}e_1(t)$ and $\lim_{t\to *\infty} e_2(t)$ exist, and are equal to each other.
\end{itemize}
\end{lem}

\begin{proof}
First of all, by property (4) in the definition of a square envelope, we immediately have that $\sigma_1^+(e_1,e_2,\gamma) = -\sigma_1^-(e_1,e_2,\gamma)$ and $\sigma_2^+(e_1,e_2,\gamma) = -\sigma_2^-(e_1,e_2,\gamma)$ because the defining loops for the parities in question compose to give us a loop that wraps around the Jordan curve exactly once.

Next, we claim that if $*\in \{+,-\}$, and the limit $\lim_{t\to *\infty} e_1(t)$ does not exist, then we have $\sigma_1^*(e_1,e_2,\gamma) = \sigma_2^*(e_1,e_2,\gamma) = 1$. To prove this claim, note that for the limit to fail to exist, there must be a limit point of $e_1(t)$, $t\to *\infty$ at one of the points where $\gamma$ is smooth. Taking a sufficiently small square of the envelope, in the $*\infty$ direction, near such a point, we see that the Jordan curve must separate the vertices of the square in a way locally equivalent to how a straight line could separate the vertices, and there is only one such way that separates the $e_1,e_2$ vertices from the other two vertices. Since all squares sufficiently near $t = \pm \infty$ will have side length smaller than the diameter of some disk inside of the Jordan curve, the paths that define $\sigma_1^*(e_1,e_2,\gamma)$ and  $\sigma_2^*(e_1,e_2,\gamma)$ cannot wrap around either of the two interior vertices of the square at $t= 0$. Therefore, we have $\sigma_1^*(e_1,e_2,\gamma) = \sigma_2^*(e_1,e_2,\gamma) = 1$. 

To prove all three parts of the lemma, all that remains is to eliminate the possibility that the limits $\lim_{t\to *\infty}e_1(t)$ and $\lim_{t\to *\infty} e_2(t)$ exist, but $\sigma_1^*(e_1,e_2,\gamma)$ and $\sigma_2^*(e_1,e_2,\gamma)$ are not equal to one another. The reason this is impossible is that for this to be the case, the paths $e_1(t), e_2(t), S_1(e_1(t),e_2(t)),$ and $S_2(e_1(t),e_2(t))$ would need to all approach some point in the plane, none of them intersecting each other, in such a way that as one encircles the point counterclockwise, the paths appear in a cyclic order that alternates between the sets $\{e_1(t), e_2(t)\}$ and $\{S_1(e_1(t),e_2(t)),S_2(e_1(t),e_2(t))\}$. This would then contradict the fact that the simple closed curve $\gamma$ separates the paths in $\{e_1(t), e_2(t)\}$ from those in $\{S_1(e_1(t),e_2(t)),S_2(e_1(t),e_2(t))\}$.
\end{proof}

If $\gamma$ is a Jordan curve which is smooth except at finitely many points, and  $(e_1,e_2)$ is a square envelope, we say  $(e_1,e_2)$ is positively oriented if $\sigma_1^+(e_1,e_2,\gamma) = -1$ and negatively oriented if $\sigma_1^+(e_1,e_2,\gamma) = +1$. Note that if $(e_1(t),e_2(t))$ is a negatively oriented square envelope, we can obtain a positively oriented one by taking $(e_1(-t), e_2(-t))$.

We need one more lemma before we can prove Theorem 3. 

\begin{lem}
Let $\alpha_1,\alpha_2, \beta_1, \beta_2$ be complex numbers such that the ratio $\beta_1/\beta_2$ is a positive real number. Let $t_0$ and $p$ be arbitrary real numbers with $p \neq 0$, and let $\lambda,r_1,r_2$ be real numbers in $[0,1)$. If there exist real numbers $t_1$ and $t_2$ in $[0,\infty)$ such that  $$e^{\alpha_1 + \beta_1p(t_1-t_0)}\left(1 + \lambda r_1e^{ip(t_1-t_0)}\right)=e^{\alpha_2 + \beta_2p(t_2-t_0)}\left(1 + \lambda r_2e^{ip(t_2-t_0)}\right)$$ then there also exist  $t_1'$ and $t_2'$ in $[0,\infty)$ such that$$e^{\alpha_1 + \beta_1p(t_1'-t_0)}\left(1 + r_1e^{ip(t_1'-t_0)}\right)=e^{\alpha_2 + \beta_2p(t_2'-t_0)}\left(1 + r_2e^{ip(t_2'-t_0)}\right)$$
\end{lem}
\begin{proof}

Making the substitutions $\alpha_1-\alpha_2 = \alpha$, $p(t_1-t_0)= t$, and $p(t_2-t_0) = (\beta_1/\beta_2)t + s$, the equation $$e^{\alpha_1 + \beta_1p(t_1-t_0)}\left(1 + \lambda r_1e^{ip(t_1-t_0)}\right)=e^{\alpha_2 + \beta_2p(t_2-t_0)}\left(1 + \lambda r_2e^{ip(t_2-t_0)}\right)$$ rearranges to $$e^{\alpha - \beta_2s}= \frac{1 + \lambda r_1e^{it}}{1 + \lambda r_2 e^{i(\frac{\beta_1}{\beta_2}t + s)}}$$
Rearranging, and setting $a = r_1$, $b(s) = -r_2e^{\alpha + (i-\beta_2)s}$, and $c(s) = e^{\alpha + \beta_2s} - 1$, we have $$ c(s) = \lambda\cdot \left( a\cdot e^{it} + b(s)\cdot  e^{i\left(\frac{\beta_1}{\beta_2}t\right)}\right) $$

Also, let $$T(s) = \{t: \exists t_1\geq 0, \exists t_2\geq 0, \;\;\; p(t_2-t_0)=(\beta_1/\beta_2)t+s, \;\;\;p(t_1-t_0) = t\}$$ be the allowable values of $t$ for fixed $s$, when $t_1$ and $t_2$ are nonnegative. This will always be an interval of the form $(-\infty,z]$ or $[z,\infty)$ for some real number $z$. Now, let $H(s)$ denote the unique minimal simply connected set containing the set $$\{a\cdot e^{it} + b(s)\cdot  e^{i(\beta_1/\beta_2)t}: t \in  T(s)\}$$ Then, we see that there exists a pair of nonnegative real numbers $t_1$ and $t_2$ such that   $$e^{\alpha_1 + \beta_1p(t_1-t_0)}\left(1 + \lambda r_1e^{ip(t_1-t_0)}\right)=e^{\alpha_2 + \beta_2p(t_2-t_0)}\left(1 + \lambda r_2e^{ip(t_2-t_0)}\right)$$ if and only if there exists a real number $s$ such that $c(s) \in \lambda H(s)$. Since $\beta_1/\beta_2$ is positive, when $\beta_1/\beta_2$ is rational, $H(s)$ is the region enclosed by the outer boundary of an epitrochoid curve.  When $\beta_1/\beta_2$ is irrational, $H(s)$ it is a disk missing some subset of its boundary. In particular, $H(s)$ is always a radial set, meaning that for any $x\in H(s)$, we have $\lambda x\in H(s)$ for all $\lambda\in [0,1]$. Therefore, increasing the parameter $\lambda$ preserves the existence of solutions to our original equation. Setting $\lambda = 1$ gives us the desired result.

\begin{figure}[H]
\centering
\includegraphics[scale = 0.3]{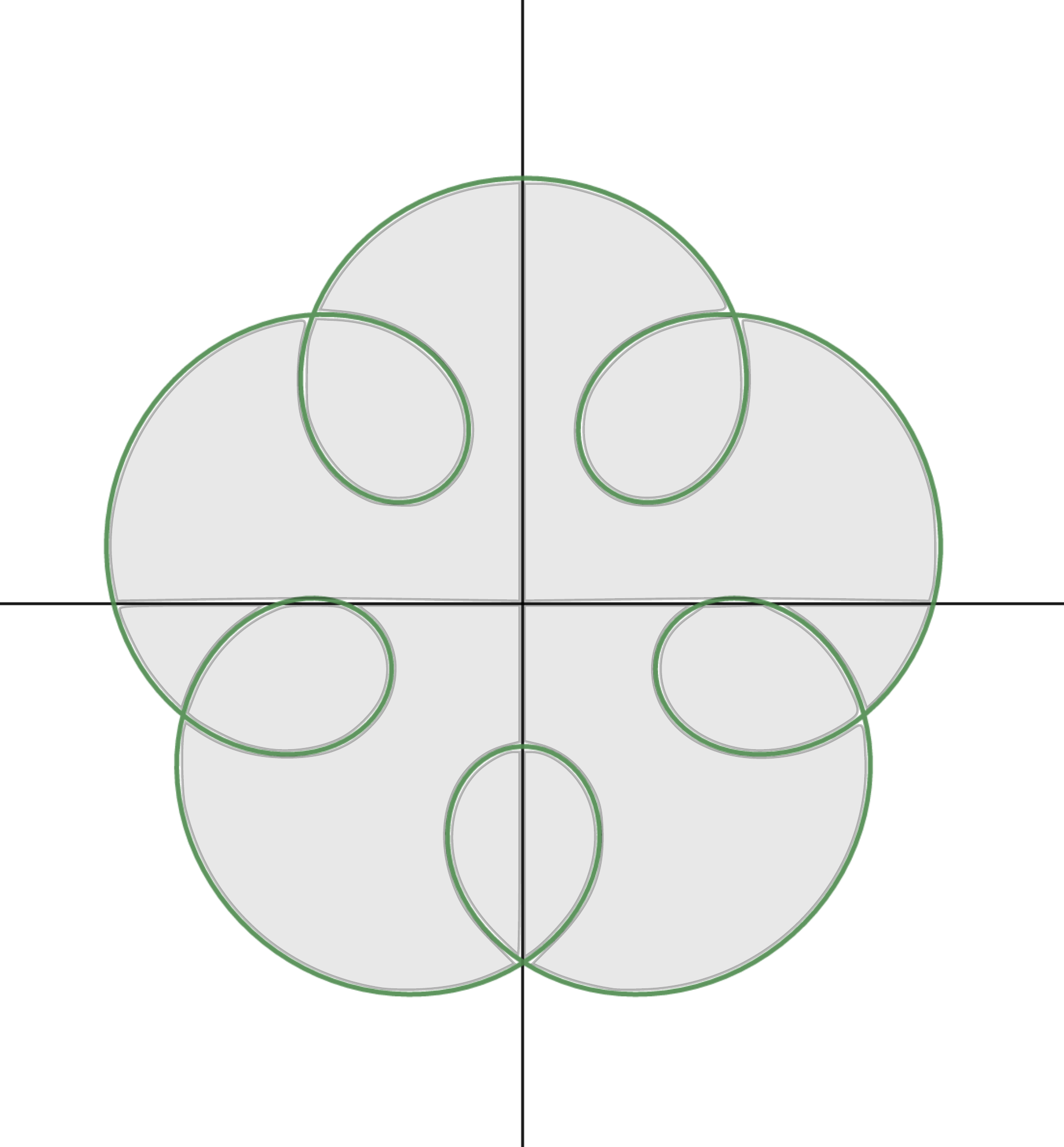}
\caption{The region enclosed by an epitrochoid curve is always a radial set.}
\end{figure}

\end{proof}

Now, we can prove Theorem 3.
\begin{proof}[Proof of Theorem 3.]
We begin by considering the system of linear relations on $\C^2$ given by letting $(x_1,y_1)R(x_2,y_2)$ when $$ \{x_1,y_1,x_2,y_2\}\cap \{x_1 + i(y_1-x_1),y_1 + i(y_1-x_1),x_2 + i(y_2-x_2),y_2 + i(y_2-x_2) \} \neq \varnothing$$ 
Now, take a positively oriented square envelope $(e_1,e_2)$. Due to Lemma 2, restricting the domain to $[0,\infty)$ immediately gives us a relation avoiding origin path $(e_1(t),e_2(t)) - \lim_{t\to \infty}(e_1(t),e_2(t))$ for $R$. We are assuming the spiral conjecture, so we may mow take $h:[0,1]\times[0,\infty)$ satisfying the conditions of the spiral conjecture. This gives us complex numbers $a_1,a_2$, and vectors $x_1,x_2\in \C^2$, such that $(f_1(t),f_2(t))= f(t) = x_1e^{a_1t} +x_2e^{a_2t}$ is a relation avoiding origin path for $R$ which has the property that the path in $\C$ starting at $0$, following $f_2$ to $f_2(0)$, going in a straight line from $f_2(0)$ to $f_1(0)$, then following $f_1$ back to $0$ has nontrivial wrapping number parity around the points $f_1(0) + i(f_2(0) - f_1(0))$  and $f_2(0) + i(f_2(0) - f_1(0))$. Without loss of generality, we can assume that $a_1$ and $a_2$ have the same real part, because otherwise one term will dominate at large $t$, reducing us to the case of a pure logarithmic spiral, which we will cover later. Similarly, we may assume $a_1\neq a_2$. Now, to prove that no such $a_1,a_2,x_1,x_2$ exist, we see that for the path to be relation avoiding, we must be able to find some fixed complex number $\beta$, and nonzero real number $p$, such that all four paths $f_1(t),f_2(t), f_1(t) + i(f_2(t) - f_1(t))$, and $f_2(t) + i(f_2(t) - f_1(t))$ are of the form $$ e^{\alpha + q \beta  p(t-t_0)}\left(1 + r e^{ip(t-t_0)}\right) $$ with $\alpha\in \C, t_0\in \R, q\in (0,\infty),$ and $r\in [0,1)$. Therefore, Lemma 3 allows us to continuously scale down all of the $r$ parameters to zero while staying a relation avoiding origin path. 

All that remains is to show that no pure logarithmic spiral $f(t) = xe^{at}$ can be a relation avoiding origin path with the stated wrapping number parity condition. To prove this, note that the region swept out by a line segment between $f_1$ and $f_2$ would have the same area as the region swept out by the line segment between $f_1 + i(f_2-f_1)$ and $f_2+i(f_2-f_1)$. However, our wrapping number condition implies that the former region would need to strictly contain the latter, which would contradict them having equal areas. 
\end{proof}

To wrap things up, we make a couple more conjectures about relation avoiding origin paths. 

\begin{cnj}[SSC, Strong Spiral Conjecture]
Every relation avoiding origin path $p$ is strongly homotopic to a relation avoiding origin path $q$ for which there exist linearly independent vectors $x_1,...,x_n$ in $V$ and complex numbers $a_1,...,a_n$ such that $$q(t) = \sum_{i = 1}^n x_i\cdot e^{a_it}$$ for all $t$.
\end{cnj}

\begin{cnj}[MSC, Monotonic Spiral Conjecture]
For any continuous function $p:[0,\infty) \to V\setminus\{0\}$ with $\lim_{t\to \infty}p(t) = 0$, there exists a continuous function $h: [0,1]\times [0,\infty) \to V$ with the following properties. 
\begin{itemize}
\item[1)]  $\lim_{t\to \infty}h(s,t) = 0$ for all $s$.
\item[2)]  $h(0,t) = p(t)$ for all $t$.
\item[3)] There exist linearly independent vectors $x_1,...,x_n$ in $V$ and complex numbers $a_1,...,a_n$ such that $$h(1,t) = \sum_{i = 1}^n x_i\cdot e^{a_it}$$ for all $t$.
\item[4)]  If $0\leq s_1\leq s_2 \leq 1$, then $B(s_2) \subseteq B(s_1)$, where $$B(s) :=  \bigcup_{(t_1,t_2) \in [0,\infty)^2}\{f\in (V\times V)': f(h(s ,t_1),h(s ,t_2)) = 0\} $$
\end{itemize} 
\end{cnj}

It is not too difficult to prove the implications $MSC \implies SSC \implies SC$. The strong spiral conjecture clearly implies the spiral conjecture, but it might not be so easy to see why the monotonic spiral conjecture implies the strong spiral conjecture. The reason is that a small perturbation of the homotopy can be made to avoid any relations with codimension greater than one in $V\times V$, and the homotopy for the monotonic spiral conjecture will always increase the set of codimension one relations that it avoids. 

Morally speaking, the monotonic spiral conjecture should be true in the one-dimensional case because curve shortening flow in a logarithmic metric should have the desired property. This argument should at least cover the case of smooth paths which are periodic deviations from a logarithmic spiral, but the analysis is more complicated for arbitrary continuous paths. It would be useful if there was a way to generalize such an argument to higher dimensions, though it is not at all obvious how to achieve this. Regardless, we hope that a deep enough understanding of relation avoiding paths will lead to new progress on the inscribed square problem.
\newpage

\bibliography{Refrences}{}
\nocite{*}
\bibliographystyle{plain}

\end{document}